\documentclass[11pt]{amsart}
\usepackage{hyperref}
\usepackage{amssymb}
\usepackage{graphicx}
\catcode`\é=\active
\defé{\'e}

\catcode`\è=\active
\defè{\`e}

\catcode`\ê=\active
\defê{\^e}

\catcode`\û=\active
\defû{\^u}

\catcode`\â=\active
\defâ{\^a}

\catcode`\à=\active
\defà{\`a}

\catcode`\ù=\active
\defù{\`u}

\catcode`\ç=\active
\defç{\c c}

\catcode`\î=\active
\defî{\^{\i}}

\catcode`\ï=\active
\defï{\"{\i}}

\catcode`\ô=\active
\defô{\^o}

\catcode`\é=\active
\defé{\'e}

\newcommand{\R}{\mathbb{R}}

\theoremstyle{plain}
\newtheorem{theorem}{Theorem}

\newtheorem{lem}{Lemma}

\begin{document}
\title[Distribution of the values of the $L'(s,\chi)$]{Distribution of the values of the derivative of the Dirichlet $L$-functions at its
$a$-points} 
\author{Mohamed Taïb Jakhlouti and Kamel Mazhouda}
\address{Mohamed Taïb Jakhlouti\\
   Faculty of Science of Monastir\\Department of Mathematics\\5000 Monastir\\Tunisia}
   \email{jmedtayeb@yahoo.com}

\address{Kamel Mazhouda\\
   Faculty of Science of Monastir\\Department of Mathematics\\5000 Monastir\\Tunisia}
   \email{kamel.mazhouda@fsm.rnu.tn}


\thanks{\today
The  authors are supported by the Tunisian-French Grant DGRST-CNRS 14/R 1501.}

\maketitle
\begin{abstract}
In this paper, we study the value distribution   of the derivative of a Dirichlet $L$-function  $L'(s,\chi)$ at the  $a$-points $\rho_{a,\chi}=\beta_{a,\chi}+i\gamma_{a,\chi}$ of $L(s,\chi).$ We give an asymptotic formula  for the
 sum $$\sum_{\rho_{a,\chi};\ 0<\gamma_{a,\chi}\leq T}L'\left(\rho_{a,\chi},\chi\right) X^{\rho_{a,\chi}}\ \ \hbox{as}\ \
 T\longrightarrow \infty,$$ where $X$ is a fixed positive number and $\chi$ is a primitive character $\mod q$.
 This work continues the investigations of  Fujii \cite{2,3,4}, Garunk$\check{s}$tis \& Steuding \cite{7} and the authors \cite{12}.
 \vskip 2cm
\end{abstract}

\section{Introduction}
 Let $L(s,\chi)$ be the Dirichlet $L$-function associated with a primitive character $\chi \mod q$ and $a$ be  a nonzero complex number. The zeros of $L(s,\chi)-a,$ which will be denoted by $\rho_{a,\chi}=\beta_{a,\chi}+i\gamma_{a,\chi},$ are called the $a$-points of
$L(s,\chi).$ First, we note that there is an $a$-point near
any trivial zero $s =-2n$  if $\chi(-1)=1$ and $s=-2n-1$ if $\chi(-1)=1$ for sufficiently large $n$. Apart from these $a$-points
there are only finitely many other $a$-points in the half-plane $Re(s)=\sigma\leq0$. The $a$-points with $\beta_{a,\chi}\leq 0$ are said to be trivial. All other $a$-points lie in a strip $0<Re(s)<A,$ where $A$ is a constant depending on $a$; these numbers are called the nontrivial $a$-points. The  number of these $a$-points satisfies  a Riemann-von Mangoldt type formula (we refer to \cite[\S 7.2]{21} for the proof of this formula which is stated for functions in a subclass of the Selberg class including the Dirichlet $L$-functions $L(s,\chi)$), namely
\begin{equation}\label{eq T2 } N_{a,\chi}(T)=\sum_{\substack{\rho_{a,\chi};\ 0<\gamma_{a,\chi}\leq T\\
\beta_{a,\chi}>0}}1 =\frac{T}{2\pi}\log\left(\frac{qT}{2\pi c_a
e}\right)+O(\log T),\end{equation} where $c_a=m$ if $a=1$ and $c_{a}=1$ otherwise, with $m=\min\{n\geq2 :\ \chi(n)\neq0\}.$
Here and in the sequel the error term depends on $q$, however, the main term is essentially independent of
$a$. Moreover, $N_{a,\chi}(T)\sim N_{\chi}(T)$ as $T\rightarrow\infty,$ where $N_{\chi}(T)=N_{0,\chi}(T)$ denotes the number of nontrivial zeros $\rho_{\chi}=\beta_{\chi}+i\gamma_{\chi}$ of $L(s,\chi)$ satisfying  $0<\gamma_{\chi}<T$.
 Gonek \cite{10}  proved that, if the Riemann Hypothesis holds for $L(s,\chi)$, then at least $\left(\frac{1}{2}+o(1)\right)\frac{T}{2\pi}\log\frac{T}{2\pi}$
of the nontrivial $a$-points with ordinates in $(0, T)$ of the function $L(s,\chi)$ associated with a primitive character $\chi$ are simple and lie to
the left of the line $Re(s)= 1/2$.\\

 In this paper, we continue the investigations of  Fujii \cite{2,3,4}, Garunk$\check{s}$tis \& Steuding \cite{7} and the authors \cite{12}. Actually, we are interested in the sums
$$\sum_{\rho_{\chi};\ 0<\gamma_{\chi}\leq T}L'\left(\rho_{\chi},\chi\right) X^{\rho_{\chi}},\ \ \sum_{\rho_{a,\chi};\ 0<\gamma_{a,\chi}\leq T}L'\left(\rho_{a,\chi},\chi\right) X^{\rho_{a,\chi}},$$ where $X$ is a
fixed positive number and $\chi$ mod $q $ is a primitive character. Our method is based on a formula stated by Garunk$\check{s}$tis and Steuding in \cite[\S 6, Remark ii)]{7} with the  function $f(s)=L'(s,\chi)X^{s}$. There are several reasons why the above sum with the parameter
$X$ is of interest. The first one is that the estimation of this sum can be used to study the normal distribution of the values of $\log\left|L'(\rho_{a,\chi},\chi)\right|$\footnote{In the case $a=0$, to study the normal distribution of the values
of $\log|\zeta'(1/2+i\gamma_{n})|$, Hiary and Odlyzko \cite{11} have been studying the behavior of the sum $\displaystyle{\sum_{T\leq\gamma_{n}\leq T+H}\zeta'(1/2+i\gamma_{n})e^{2\pi inx}}$ as a function of $x$, where $\rho=\beta+i\gamma$ denotes a non-trivial zero of $\zeta(s)$ and $\gamma_{n}$ is the $n$th positive imaginary part of a zero $\rho$. To do so, they  approximated the last sum by $\displaystyle{\sum_{T\leq\gamma_{n}\leq T+H}\zeta'(1/2+i\gamma_{n})e^{2\pi i\tilde{\gamma}_{n}x}}$, where $\tilde{\gamma}_{n}=\frac{1}{2\pi}\gamma_{n}\log\frac{T}{2\pi}$.}
, the second one is  to study the vertical distribution of $a$-points of $L(s,\chi)$.
Recall that in the case of $a = 0$, recently Fujii studied in \cite{4}  sums over the nontrivial zeros of $L(s,\chi)$.
 He  showed that under the Riemann hypothesis, for $X>1$,
$$\lim_{T\rightarrow\infty}\frac{1}{T/2\pi}\sum_{0<\gamma\leq T}\left[X^{1/2+i\gamma}\left(L(1/2+i\gamma,\chi)-1\right)-\xi(X)\right]=\left\{\begin{array}{crll}M(X,\chi)&\hbox{if}&X \hbox{\  is rational,}\\
0&\hbox{if}&X \hbox{\ is irrational,}
\end{array}\right.$$
where $L(s,\chi)$ is a Dirichlet $L$-function with primitive Dirichlet character $\chi(\mod\ q\geq3)$, $\xi(X)$ and $M(X,\chi)$ are some constants.
Furthermore, Garunk$\check{s}$tis,  Grahl and  Steuding \cite{8} obtained more subtle information on the value distribution of Dirichlet $L$-functions by considering certain discrete moments $\displaystyle{\sum_{\rho_{a,\chi};\ 0<\gamma_{a,\chi}<T}L(\rho_{a,\chi},\psi)}$. 
Their  formula extends a previous result due to Fujii \cite{5,6}. \\

Our main  result is stated in the following :
\begin{theorem}\label{theorem}
Let $X$ be a positive number. Then
\begin{eqnarray}  \label{eql} && \sum_{\rho_{a,\chi};\ 0<\gamma_{a,\chi}\leq
T}L'\left(\rho_{a,\chi},\chi\right)X^{\rho_{a,\chi}}  \nonumber\\
&=& -\frac{a T}{2\pi}\log^{2}\left(\frac{q T}{2\pi}\right)+\frac{aT}{\pi}\log\left(\frac{q  T}{2\pi}\right)-\frac{aT}{\pi}\nonumber\\
&&\  -\     \Delta(X)\chi\left[\Delta(X)X\right]\log{(X)}\left\{\frac{T}{4\pi}\log{\left(\frac{q
T}{2\pi}\right)}-\frac{T}{4\pi}+\frac{i\pi}{4}\frac{T}{2\pi}\right\}\nonumber\\
&&\ \ +\ \  \Delta(X)\chi\left[\Delta(X)X\right]\frac{T}{2\pi}\sum_{X=m  n}\Lambda{(n)}\log{(m)}\nonumber\\
&&\ \ +\ \ \frac{X}{\sqrt{q}}\sum_{k\leq  \frac{q T}{2\pi X}}\log^2(k)\overline{\chi}(k)e^{2i\pi k X/q}+\frac{1}{2\sqrt{q}}X\log(X)\sum_{k\leq \frac{q T}{2\pi
X}}\log(k)\overline{\chi}(k)e^{2i\pi k X/q}\nonumber\\
&&\ \ -\ \   \left(\frac{1}{2\sqrt{q}}X\log^2(X)-\frac{i\pi}{4\sqrt{q}}X\log{X}\right)\sum_{k\leq \frac{q T}{2\pi
X}}\overline{\chi}(k)e^{2i\pi k X/q} \nonumber\\
&&\ \ -\ \ \frac{X}{\sqrt{q}}\sum_{k\leq \frac{q T}{2\pi
X}}\sum_{k=mn}\Lambda(n)\overline{\chi}(n)\overline{\chi}(m)\log(m)e^{2i\pi k X/q}+ O\left(\sqrt{T}\log^3 T\right),
\end{eqnarray}
where
$\Delta(X) $ is defined by
\begin{equation}\label{delta}
\Delta(X)=\left\{
    \begin{array}{ll}
      1 & \hbox{if X is an integer }\geq 1, \\
      0 & \hbox{otherwise.}
    \end{array}
  \right.
  \end{equation}
\end{theorem}
\noindent{\bf Remark.} For $X=1$, we obtain Garunk$\check{s}$tis and  Steuding's results \cite{7} in the case of the  Riemann zeta function. And for $q=1$ and $a=0$, we obtain Fujii's results \cite{3}.\\

Here and in the sequel the implicit constant in the error terms may depend on $a$ and $X$; the formulas of   Theorem \ref{theorem}  are
     not uniform with respect to $X$.  We note that the proofs  uses standard methods: contour integration, basic properties of the
functional equation for $L(s,\chi)$ and  Gonek's lemma (see. Gonek  \cite[Lemma 5]{9}).

\section{Preliminary lemmas}
To prove Theorem \ref{theorem}, we start with well-known results on the Dirichlet $L$-function $L(s,\chi)$ (see Davenport book \cite{1}). If $\chi$ mod $q$ is a primitive character, then $$\xi(s,\chi)=\left(\frac{q}{\pi}\right)^{\frac{s+\nu}{2}}\Gamma\left(\frac{s+\nu}{2}\right)L(s,\chi)$$ satisfies the functional equation

\begin{equation}\label{eqqL1}
\xi(s,\chi)=\frac{\tau(\chi)}{i^{\nu}\sqrt{q}}\xi\left(1-s,\overline{\chi}\right)
\end{equation}
where $\tau(\chi)=\sum_{m \mod q}\chi(m)e^{\frac{2i\pi m}{q}},$
with $\nu=\frac{1}{2}(1-\chi(-1)).$ We note that
$$\frac{\xi'}{\xi}(s,\chi)=\frac{1}{2}\log\left(\frac{q}{\pi}\right)+\frac{1}{2}\psi\left(\frac{s+\nu}{2}\right)+\frac{L'}{L}(s,\chi),$$
where $\psi(s)=\frac{\Gamma'}{\Gamma}(s)$ and that for
$|\arg{(s)}|<\pi-\theta$ with arbitrary fixed positive $\theta$ and
for $|s|\geq \frac{1}{2},$ we have \begin{equation}\label{eqL2}
\psi(s)=\log{(s)}+O\left(\frac{1}{|s|}\right)=\log{|t|}+\frac{i\pi}{2}+O\left(\frac{\sigma}{|t|}\right),
\end{equation}
as $|t|\longrightarrow \infty.$ In \cite[Lemma 8]{8}, Garunk$\check{s}$tis,  Grahl and  Steuding  proved that there exist positive constants $c_{1}$ and $c_{2}$ such that, for $\sigma\leq 0$ and $|t|>2$,
\begin{equation}\label{eq.moment1}|L(\sigma+it,\chi)|>\frac{c_{1}|t|^{\frac{1}{2}-\sigma}}{\log^{7}t}\end{equation}
and
\begin{equation}\label{eq.moment2}|L(\sigma+it,\chi)|<c_{2}|t|^{\frac{1}{2}-\sigma}\log t.\end{equation}
Furthermore, for $t>t_0$ and
$1-\frac{c}{\log{(t)}}\leq \sigma\leq 2$, we have (see. \cite[page. 28]{8})
\begin{equation}\label{eqL3}
L'(s,\chi)=O\left(\log^2 t\right).
\end{equation}
Using partial summation and the P\'olya-Vinogradov inequality,  for $t\geq t_0>0$ and for any $\sigma>1$, we obtain ¨
\begin{equation}\label{eqL4} L'(s,\chi)=-\sum_{n\leq t}\frac{\chi(n)\log(n)}{n^s}+O\left(t^{-\sigma}\log t\right).
\end{equation}
From the partial fraction decomposition of $L(s,\chi)$, we get (see Davenport book \cite{1} or \cite[Equation (19)]{8})
\begin{equation}\label{eq.daven}
\frac{L'}{L}(\sigma+it, \chi)\ll \log^{2}|t|, \ \ \ \hbox{for} \ -1\leq\sigma\leq2\ \hbox{and}\ |t|\geq2.
\end{equation}
By the functional equation and the Phragmén-Lindel$\ddot{o}$f principle,  we deduce that
$$L(\sigma+it,\chi)\ll_{\epsilon}\left\{\begin{array}{crll} |t|^{\frac{1}{2}-\sigma+\epsilon}& \hbox{if} &\sigma<0,\\ |t|^{\frac{1}{2}(1-\sigma)+\epsilon}& \hbox{if} &0\leq\sigma\leq1,\\ |t|^{\epsilon}&\hbox{if} &\sigma>1,\end{array}\right.$$
as $|t|\rightarrow\infty$ and where $\epsilon$ is an arbitrarily small positive number (this is a special case of \cite[Lemma 2.1]{16} established for functions in the Selberg class in which the Dirichlet $L$-functions are elements).  Then, by Cauchy's integral formula, we get
          $$L'(s,\chi)=\frac{1}{2\pi i}\int_{{\mathcal L}}\frac{L(s,\chi)}{(\omega-s)^{2}}d\omega,$$
          \noindent          where  ${\mathcal L}$ is any arbitrarily small circle with center $s$. Using the last bound of $ L(\sigma+it,\chi)$, it follows that
                  $$L'(\sigma+it,\chi)\ll_{\epsilon}\left\{\begin{array}{crll} |t|^{\frac{1}{2}-\sigma+\epsilon}& \hbox{if} &\sigma<0,\\
                  |t|^{\frac{1}{2}(1-\sigma)+\epsilon}& \hbox{if} &0\leq\sigma\leq1,\\ |t|^{\epsilon}&\hbox{if} &\sigma>1.\end{array}\right.$$
Furthermore, for fixed complex number $a$, for \ $-1\leq\sigma\leq2$ and $|t|\geq1$, we have
\footnote{The proof is very closely to that stated in \cite[Lemma 8]{7}  with minor change  and using \cite[Ch. 7, Theorem 4.1]{18} (see also \cite[Lemmas 2.4 and 2.6]{17}).}
$$\frac{L'(s,\chi)}{L(s,\chi)-a}=\sum_{|t-\gamma_{a,\chi}|\leq1}\frac{1}{s-\rho_{a,\chi}}+O\left(\log q(|t|+1)\right).$$
Let $b$ be some constant which will be given below. In view of the number of nontrivial $a$-points  \eqref{eq T2 }, we obtain for $\sigma>1-b$
\begin{equation}\label{eq.l'/l}
\frac{L'^{2}(\sigma+it,\chi)}{L(\sigma+it,\chi)-a}\ll |t|^{(1-\sigma)/2+\epsilon}, \ \ \hbox{as}\ \ |t|\geq2.
\end{equation}
Let
$$\Delta(s,\chi)=\frac{\tau(\chi)}{i^{\nu}\sqrt{\pi}}\left(\frac{\pi}{q}\right)^{s}\frac{\Gamma\left(\frac{1}{2}(1-s+\nu)\right)}{\Gamma\left(\frac{1}{2}(s+\nu)\right)}.$$
In the next lemma, we obtain the approximate functional equation of $L'(s,\chi)$ in the following form (which  will be sufficient for our purpose).
\begin{lem}\label{lemma.L'}
For $t>t_{0}>0$ and  $0\leq \sigma\leq1$, we have
\begin{eqnarray*}
L'(s,\chi)&=&-\sum_{n\leq\sqrt{\frac{qt}{2\pi}}}\frac{\chi(n)\log(n)}{n^s}-\log\left(\frac{qt}{2\pi}\right)\Delta(s,\chi)\sum_{n\leq
\sqrt{\frac{qt}{2\pi}}}\frac{\overline{\chi(n)}}{n^{1-s}}\\
&&\ \ +\ \
\Delta(s,\chi)\sum_{n\leq\sqrt{\frac{qt}{2\pi}}}\frac{\overline{\chi(n)}\log{n}}{n^{1-s}}+O\left(t^{-\frac{\sigma}{2}}\log{t}\right).
\end{eqnarray*}
\end{lem}
\begin{proof} The proof uses the same argument and similar notations as Levinson \cite{15}. According to Lavrik \cite[Corollary 1 of Theorem 1]{13} or to Rane \cite{19}, we  use the following approximate functional equation of $L(s,\chi)$, if $\chi$ mod
$q$ is a primitive character, we have \footnote{An exact expression for the error term in the approximate functional equation of $L(s,\chi)$
is stated in \cite[Theorem 1 and Lemma 5]{13} and which noted  $R_{xy}$ (see also \cite[page 141]{19}). For example, with Lavrik's notation, when $x=y=\sqrt{\frac{qt}{2\pi}}$, one can see that the error term
$R_{xy}$ satisfies $R_{xy}=O_{q}\left(t^{-\frac{\sigma}{2}}\right)$.}
$$L(s,\chi)=\sum_{n\leq
\sqrt{\frac{qt}{2\pi}}}\frac{\chi(n)}{n^s}+\Delta(s,\chi)\sum_{n\leq
\sqrt{\frac{qt}{2\pi}}}\frac{\overline{\chi}(n)}{n^{1-s}}+O\left(t^{-\frac{\sigma}{2}}\right),$$
where $\Delta(s,\chi)$ can be written as follows
$\Delta(s,\chi):=i\tau(\chi)\chi(-1)(2\pi)^{s-1}q^{-s}\Gamma(1-s)e^{-\frac{i\pi
s}{2}}.$ Writing $L(s,\chi)=f_1(s)+\Delta(s,\chi)f_2(s)$ with $\displaystyle{ f_1(s)=\sum_{n\leq
\sqrt{\frac{qt}{2\pi}}}\frac{\chi(n)}{n^s}+O\left(t^{-\frac{\sigma}{2}}\right)}$
and
$$\Delta(s,\chi)f_2(s)=L(s,\chi)-f_1(s)=\Delta(s,\chi)\sum_{n\leq
\sqrt{\frac{qt}{2\pi}}}\frac{\overline{\chi}(n)}{n^{1-s}}.$$
Similarly to \cite[page 389]{15}, we get \footnote{Futhermore, from \cite[Theorem 1 and Lemma 5]{13} the  expression $R_{xy}$ is differentiable
and $\frac{dR_{xy}}{d\sigma}=O_{q}\left(t^{-\frac{\sigma}{2}}\log{(t)}\right)$.}
\begin{eqnarray*}
L'(s,\chi)&=&f_1'(s)+\Delta'(s,\chi)f_2(s)+\Delta(s,\chi)f_2'(s)\\
&=&-\sum_{n\leq
\sqrt{\frac{qt}{2\pi}}}\frac{\chi(n)\log(n)}{n^s}+O\left(t^{-\frac{\sigma}{2}}\log{(t)}\right)\\
&&\ \ +\ \ \Delta'(s,\chi)f_2(s)+\Delta(s,\chi)\left[\sum_{n\leq
\sqrt{\frac{qt}{2\pi}}}\frac{\overline{\chi}(n)\log(n)}{n^{1-s}}\right].
\end{eqnarray*}
Since
$$\Delta'(s,\chi)f_2(s)=\frac{\Delta'}{\Delta}(s,\chi)\Delta(s,\chi)f_2(s)=\frac{\Delta'}{\Delta}(s,\chi)\left[\Delta(s,\chi)\sum_{n\leq
\sqrt{\frac{qt}{2\pi}}}\frac{\overline{\chi}(n)}{n^{1-s}}\right].
$$
Hence, by using that for $t>t_{0}$
\begin{equation}\label{eqq1}\frac{\Delta'}{\Delta}(\sigma+it,\chi)=-\log\left(\frac{q
t}{2\pi}\right)+O\left(\frac{1}{t}\right),\end{equation} we finish the proof of Lemma \ref{lemma.L'}
\end{proof}
Using the   approximate functional equation of $L'(s,\chi)$ given in Lemma \ref{lemma.L'}, we  prove easily with the same argument used by Fujii in \cite[Lemma 3]{3} the following result.
\begin{lem}\label{lemm} Let $\delta=1+\frac{1}{\log{T}}$. Then
\begin{equation}\label{eqLem1}\int_{1-\delta}^\delta|L'(\sigma+iT,\chi)|d\sigma\ll
\sqrt{T}\log T.
\end{equation}
\end{lem}
 \begin{proof}
 From the functional equation of $L(s,\chi)$, we have
 $$L'(s,\chi)=\frac{1}{\Delta(1-s,\overline{\chi})}\left(-L'(1-s,\overline{\chi})+\frac{\Delta'}{\Delta}(1-s,\overline{\chi})L(1-s,\overline{\chi})\right).$$
 Therefore
 $$\int_{1-\delta}^\delta|L'(\sigma+iT,\chi)|d\sigma=\int_{1-\delta}^{1/2}|L'(\sigma+iT,\chi)|d\sigma+\int_{1/2}^{\delta}|L'(\sigma+iT,\chi)|d\sigma=M_{1}+M_{2},$$
 where
 \begin{eqnarray}M_{1}&=&\int_{1-\delta}^{1/2}\left|\frac{1}{\Delta(1-\sigma-iT,\overline{\chi})}\left(-L'(1-\sigma-iT,\overline{\chi})+\frac{\Delta'}{\Delta}(1-\sigma-iT,\overline{\chi})L(1-\sigma-iT,\overline{\chi})\right)\right|d\sigma\nonumber\\
 &=&\int_{1/2}^{\delta}\left|\frac{1}{\Delta(\sigma+iT,\overline{\chi})}\left(-L'(\sigma+iT,\overline{\chi})+\frac{\Delta'}{\Delta}(\sigma+iT,\overline{\chi})L(\sigma+iT,\overline{\chi})\right)\right|d\sigma.\nonumber
 \end{eqnarray}
 In any strip $\sigma_{1}\leq\sigma\leq\sigma_{2}$, we have uniformly as $t\rightarrow\infty$, \ $\left|\frac{1}{\Delta(\sigma+iT,\overline{\chi})}\right|\ll\left(\frac{qT}{2\pi}\right)^{\sigma-1/2}.$ Applying the last asymptotic formula, we obtain
 $$M_{1}\ll \int_{1/2}^{\delta}|L'(\sigma+iT,\overline{\chi})|\left(\frac{qT}{2\pi}\right)^{\sigma-1/2}d\sigma+\log T \int_{1/2}^{\delta}|L(\sigma+iT,\overline{\chi})|\left(\frac{qT}{2\pi}\right)^{\sigma-1/2}d\sigma.$$
 Then
 $$M_{1}+M_{2}\ll \int_{1/2}^{\delta}|L'(\sigma+iT,\overline{\chi})|\left(\frac{qT}{2\pi}\right)^{\sigma-1/2}d\sigma+\log T \int_{1/2}^{\delta}|L(\sigma+iT,\overline{\chi})|\left(\frac{qT}{2\pi}\right)^{\sigma-1/2}d\sigma.$$
 Let write the first integral as follows
\begin{eqnarray}\int_{1/2}^{\delta}&|L'(\sigma+iT,\overline{\chi})|&\left(\frac{qT}{2\pi}\right)^{\sigma-1/2}d\sigma=\int_{1/2}^{1}|L'(\sigma+iT,\overline{\chi})|\left(\frac{qT}{2\pi}\right)^{\sigma-1/2}d\sigma\nonumber\\&&\ +\ \int_{1}^{\delta}|L'(\sigma+iT,\overline{\chi})|\left(\frac{qT}{2\pi}\right)^{\sigma-1/2}d\sigma=M_{3}+M_{4}.\nonumber
\end{eqnarray}
Using the approximate functional equation of $L'(s,\chi)$ and the fact that $|\overline{\chi}(n)|\leq1$, we get
{\small\begin{eqnarray}&&M_{3}\nonumber\\&\ll&\int_{1/2}^{1}\left|\sum_{n\leq\sqrt{\frac{qt}{2\pi}}}\frac{\log(n)}{n^s}+T^{\frac{1}{2}-\sigma}\sum_{n\leq
\sqrt{\frac{qt}{2\pi}}}\frac{1}{n^{1-s}} +T^{\frac{1}{2}-\sigma}\log T\sum_{n\leq\sqrt{\frac{qt}{2\pi}}}\frac{\log{n}}{n^{1-s}}+O\left(T^{-\frac{\sigma}{2}}\log T\right)\right|T^{\sigma-\frac{1}{2}}d\sigma\nonumber\\
&\ll&\sqrt{T}\log T.\nonumber
\end{eqnarray}}
Now, using another approximation of $L'(s,\chi)$ as given by equation \eqref{eqL4} above, we get
\begin{eqnarray}
M_{4}&\ll&\int_{1}^{\delta}\left(\sum_{n\leq T}\frac{|\chi(n)|\log(n)}{n^\sigma}+O\left(T^{-\sigma}\log T\right)\right)T^{\sigma-\frac{1}{2}}d\sigma\nonumber\\
&\ll&\sqrt{T}(\delta-1)\sum_{n\leq T}\frac{\log(n)}{n}\nonumber\\
&\ll&\sqrt{T}\log T.\nonumber
\end{eqnarray}
Similarly, we get
$$\int_{1/2}^{\delta}|L(\sigma+iT,\overline{\chi})|\left(\frac{qT}{2\pi}\right)^{\sigma-1/2}d\sigma\ll\sqrt{T}.$$
Hence, we obtain the assertion of Lemma \ref{lemm}.
 \end{proof}
 An explicit formula for the sums $$\displaystyle{\sum_{\rho_{\chi}=\beta_{\chi}+i\gamma_{\chi};\ 0<\gamma_{\chi}\leq T}L'\left(\rho_{\chi},\chi\right) X^{\rho_{\chi}}},$$ where $\rho_{\chi}$  runs over the nontrivial zeros of $L(s,\chi)$ is stated in the following :
\begin{lem}\label{lem.2} Let $X$ be a  positive number and $\rho_{\chi}=\beta_{\chi}+i\gamma_{\chi}$ denotes a nontrivial zero of the
Dirichlet $L$-function $L(s,\chi)$. Then
\begin{eqnarray}\label{eql} &&\sum_{\rho_{\chi};\ 0<\gamma_{\chi}\leq
T}L'\left(\rho_{\chi},\chi\right)X^{\rho_{\chi}}
 =-\Delta(X)\chi\left[\Delta(X)X\right]\log{(X)}\left\{\frac{T}{4\pi}\log{\left(\frac{q
T}{2\pi}\right)}-\frac{T}{4\pi}+\frac{i\pi T}{8\pi}\right\}\nonumber\\
&&\ \ \ \ \ \ \ \ +\ \  \Delta(X)\chi\left[\Delta(X)X\right]\frac{T}{2\pi}\sum_{X=m
n}\Lambda{(n)}\log{(m)}\nonumber\\
&&\ \ \ \ \ \ \ \ +\ \ \frac{X}{\sqrt{q}}\sum_{k\leq \frac{q T}{2\pi X}}\log^2(k)\overline{\chi}(k)e^{2i\pi k X/q}
+\frac{X\log(X)}{2\sqrt{q}}\sum_{k\leq \frac{q T}{2\pi X}}\log(k)\overline{\chi}(k)e^{2i\pi k X/q}\nonumber\\ &&
\ \ \ \ \ \ \ \ -\ \
\frac{1}{\sqrt{q}}\left(\frac{1}{2}X\log^2(X)-\frac{i\pi}{4}X\log{X}\right)\sum_{k\leq
\frac{q T}{2\pi X}}\overline{\chi}(k)e^{2i\pi k X/q}\nonumber\\
&&\ \ \ \ \ \ \ \ \  -\ \frac{X}{\sqrt{q}}\sum_{k\leq \frac{q T}{2\pi
X}}\sum_{k=mn}\Lambda(n)\overline{\chi}(n)\overline{\chi}(m)\log(m)e^{2i\pi k X/q}+
O\left(\sqrt{T}\log^3{T}\right),\end{eqnarray} where
$\Delta(X) $ is defined by equation \eqref{delta}.
\end{lem}
\begin{proof} We apply the same argument used by Fujii in \cite{3}. Let $X$ be a fixed positive number, $s=\sigma+it, t\in \R.$ Suppose
that $T>t_0$, where $t_{0}>0$ is as in equations \eqref{eqL3} or \eqref{eqL4} and $T$ is not an imaginary part of any  zero of   the
Dirichlet $L$-function. We consider
\begin{equation}\label{eqL 1}I=\frac{1}{2 i\pi}\int_{\mathbf{R}}\frac{\xi '}{\xi}(s,\chi)L'(s,\chi)X^s
ds, \end{equation} where $\mathbf{R}$ denotes the counterclockwise
oriented rectangular with vertices $\delta+iC,\ \delta+iT$,\\
$1-\delta+iT$ and $1-\delta+iC,$ with $\delta=1+\frac{1}{\log{T}}$.
First, we have
\begin{eqnarray}
I&&\ =\ \frac{1}{2i\pi}\int_{\delta+iC}^{\delta+iT}\frac{\xi'}{\xi}(s,\chi)L'(s,\chi)X^sds+\frac{1}{2i\pi}\int_{\delta+iT}^{1-\delta+iT}\frac{\xi'}{\xi}(s,\chi)L'(s,\chi)X^sds\nonumber\\
&&\ \ \ +\ \
\frac{1}{2i\pi}\int_{1-\delta+iT}^{1-\delta+iC}\frac{\xi'}{\xi}(s,\chi)L'(s,\chi)X^sds+\frac{1}{2i\pi}\int_{1-\delta+iC}^{\delta+iC}\frac{\xi'}{\xi}(s,\chi)L'(s,\chi)X^sds\nonumber\\
&&\ =\ I_1+I_2+I_3+I_4.\label{eqq2}
\end{eqnarray}
Recall that from our choice of $T$, we have
\begin{equation}\label{eqq3}
\frac{\xi'}{\xi}(\sigma+it,\chi)\ll \log^2 T, \ \ \ \ \
\hbox{for}\ \ \ -1\leq \sigma\leq 2.
\end{equation}
Therefore, by using Lemma \ref{lemm}, we deduce that
\begin{equation}\label{eqq4}
I_2+I_4\ll
\log^{2}T\int_{1-\delta}^{\delta}\left(|L'(\sigma+iC,\chi)+|L'(\sigma+iT,\chi)|\right)X^{\sigma}d\sigma\ll\sqrt{T}\log^3 T.
\end{equation}
Now, we estimate $I_1.$ We have \begin{eqnarray}
I_1&=&\frac{1}{2\pi}\int_C^T\frac{\xi'}{\xi}(\delta+it,\chi)L'(\delta+it,\chi)X^{\delta+it}dt\nonumber\\
&=&
\frac{1}{2\pi}\int_C^T\left\{\frac{1}{2}\log{\left(\frac{q}{\pi}\right)}+\frac{1}{2}\psi\left(\frac{\delta+\nu+it}{2}\right)+\frac{L'}{L}(\delta+it,\chi)\right\}L'(\delta+it,\chi)X^{\delta+it}dt\nonumber
\end{eqnarray}
Formula (\ref{eqL2}) yields to
\begin{eqnarray}
&&I_1\nonumber\\
&=&-\frac{1}{2\pi}\int_C^T\left\{\frac{1}{2}\log\left(\frac{qt}{2\pi}\right)+\frac{i\pi}{4}-\sum_{n=1}^{\infty}\frac{\chi(n)\Lambda(n)}{n^{\delta+it}}+O\left(\frac{1}{t}\right)\right\}\sum_{m=1}^{\infty}\frac{\chi(m)\log{(m)}}{m^{\delta+it}}X^{\delta+it}dt\nonumber\\
&=&-\frac{X^{\delta}}{2\pi}\sum_{m=1}^{\infty}\frac{\chi(m)\log{(m)}}{m^{\delta}}\int_C^T\left\{\frac{1}{2}\log\left(\frac{qt}{2\pi}\right)+\frac{i\pi}{4}+O\left(\frac{1}{t}\right)\right\}\left(\frac{X}{m}\right)^{it}dt\nonumber\\
&&\ \ \ \ \ \ +\ \ \frac{X^{\delta}}{2\pi}\sum_{m=1}^{\infty}\frac{\chi(m)\log{(m)}}{m^{\delta}}\sum_{n=1}^{\infty}\frac{\chi(n)\Lambda(n)}{n^{\delta}}\int_C^T\left(\frac{X}{m n}\right)^{it}dt=\  J_1+J_2\label{eqq5},
\end{eqnarray}
where
\begin{eqnarray}J_1&=&-\frac{X^{\delta}}{2\pi}\sum_{m=1}^{\infty}\frac{\chi(m)\log{(m)}}{m^{\delta}}\int_C^T\left\{\frac{1}{2}\log\left(\frac{qt}{2\pi}\right)+\frac{i\pi}{4}+O\left(\frac{1}{t}\right)\right\}\left(\frac{X}{m}\right)^{it}dt\nonumber\\
&=&-\Delta(X)\frac{\chi\left[\Delta(X)X\right]\log{(X)}}{2\pi}\int_C^T\left\{\frac{1}{2}\log\left(\frac{qt}{2\pi}\right)+\frac{i\pi}{4}\right\}dt\nonumber\\
&&\ \ -\ \ \frac{X^{\delta}}{2\pi}\sum_{\substack{m=1\\ m\neq
X}}^{\infty}\frac{\chi(m)\log(m)}{m^{\delta}}\int_C^T\left\{\frac{1}{2}\log\left(\frac{qt}{2\pi}\right)+\frac{i\pi}{4}\right\}\left(\frac{X}{m}\right)^{it}dt+O\left(\log^3 T\right)\nonumber\\
&=&J_3+J_4+O\left(\log^3 T\right)\nonumber
\end{eqnarray}
with
\begin{eqnarray}
J_3&=&\ -\Delta(X)\frac{\chi\left[\Delta(X)X\right]\log{(X)}}{2\pi}\int_C^T\left\{\frac{1}{2}\log\left(\frac{qt}{2\pi}\right)+\frac{i\pi}{4}\right\}dt\nonumber\\
&=&\ -\Delta(X)\chi\left[\Delta(X)X\right]\log{(X)}\left\{\frac{T}{4\pi}\log{\left(\frac{qT}{2\pi}\right)}-\frac{T}{4\pi}+\frac{i\pi}{4}\frac{T}{2\pi}\right\}+O(1)\nonumber
\end{eqnarray}
and
\begin{eqnarray}
J_4&=&-\ \ \frac{X^{\delta}}{2\pi}\sum_{\substack{m=1\\ m\neq
X}}^{\infty}\frac{\chi(m)\log(m)}{m^{\delta}}\int_C^T\left\{\frac{1}{2}\log\left(\frac{qt}{2\pi}\right)+\frac{i\pi}{4}\right\}\left(\frac{X}{m}\right)^{it}dt\nonumber\\
&\ll& \left\{
                               \begin{array}{ll}
                                X^{\delta}\sum_{\substack{m=1\\ m\neq
X}}^{\infty}\frac{\log m}{m^{\delta}}\min\left(T\log T,\log T\left(1+\left|\log\frac{X}{m}\right|^{-1}\right)\right) & \hbox{if}\ \ X\geq 1, \\
                                 \sum_{m=1}^{\infty}\frac{\log m}{m^{\delta}} &
\hbox{if}\ \ 0<X<1.
                               \end{array}
                             \right.\nonumber
\end{eqnarray}
Hence
$$J_4\ll \left\{
        \begin{array}{ll}
          \log^3{T} & \hbox{if}\ \ X\geq 1 ,\\
          \log{T} & \hbox{if}\ \ 0<X<1.
        \end{array}
      \right.$$
\noindent Therefore
 \begin{equation}\label{eqq6}
J_1=-\Delta(X)\chi\left[\Delta(X)X\right]\log{(X)}\left\{\frac{T}{4\pi}\log{\left(\frac{qT}{2\pi}\right)}-\frac{T}{4\pi}+\frac{i\pi}{4}\frac{T}{2\pi}\right\}+O\left(\log^3{T}\right).
\end{equation}
On the other hand, we have
\begin{eqnarray}
J_2&=&\frac{X^{\delta}}{2\pi}\sum_{k=1}^{\infty}\frac{\chi(k)}{k^{\delta}}\sum_{k=mn}\Lambda(n)\log(m)\int_C^T\left(\frac{X}{k}\right)^{it}dt\nonumber\\
&=&\frac{X^{\delta}}{2\pi}T\sum_{\substack{k=1\\
X=k}}^{\infty}\frac{\chi(k)}{k^{\delta}}\sum_{X=mn}\Lambda(n)\log(m)\nonumber\\
&&\ \ +\ \  \frac{X^{\delta}}{2\pi}\sum_{\substack{k=1\\ k\neq X}}^{\infty}\frac{\chi(k)}{k^{\delta}}\sum_{k=mn}\Lambda(n)\log(m)\int_C^T\left(\frac{X}{k}\right)^{it}dt=J_5+J_6,\nonumber
\end{eqnarray}
where
\begin{equation}\label{eq.j5}J_5=\Delta(X)\frac{T}{2\pi}\chi\left[\Delta(X)X\right]\sum_{X=mn}\Lambda(n)\log(m)+O(1)\end{equation}
and \begin{eqnarray}\label{eq.j6}
J_6&\ll&\left\{              \begin{array}{ll}
               X^{\delta}\sum_{\substack{k=1\\
X\neq k}}\frac{1}{k^{\delta}}\sum_{k=mn}\Lambda(n)\log(m)\min\left(T,\frac{1}{\left|\log{\frac{X}{k}}\right|}\right)  & \hbox{if}\ \ X\geq 1 ,\\
             \sum_{k=1}^{\infty}\frac{1}{k^{\delta}}\sum_{nm=k}\Lambda(n)\log(m)\frac{1}{\log{k}}    & \hbox{if}\ \ 0<X<1,
              \end{array}
            \right.\nonumber\\
&&\ \ \ \ \ll \left\{
        \begin{array}{ll}
          \log^3{T} & \hbox{if}\ \ X\geq 1 ,\\
          \log^2{T} & \hbox{if}\ \ 0<X<1.
        \end{array}
      \right.
\end{eqnarray}
Combining the two last equations \eqref{eq.j5} and \eqref{eq.j6}, we get  \begin{equation}\label{eqq7}
J_2=\Delta(X)\frac{T}{2\pi}\chi\left[\Delta(X)X\right]\sum_{X=mn}\Lambda(n)\log(m)+O\left(\log^3{T}\right).
\end{equation}
Therefore, from (\ref{eqq6}) and
(\ref{eqq7}), we obtain
\begin{eqnarray}
I_1&=&-\Delta(X)\chi\left[\Delta(X)X\right]\log{(X)}\left\{\frac{T}{4\pi}\log{\left(\frac{q T}{2\pi}\right)}-\frac{T}{4\pi}+\frac{i\pi}{4}\frac{T}{2\pi}\right\}\nonumber\\
&&\ \ +\ \
\Delta(X)\frac{T}{2\pi}\chi\left[\Delta(X)X\right]\sum_{X=mn}\Lambda(n)\log(m)+O\left(\log^3{T}\right).\label{eqq8}
\end{eqnarray}
Next, we shall evaluate  $I_3.$ First, we note that
\begin{eqnarray}
L'(s,\chi)&=&\left(\frac{1}{\Delta\left(1-s,\overline{\chi}\right)}L\left(1-s,\overline{\chi}\right)\right)'\nonumber\\
&=&\frac{1}{\Delta\left(1-s,\overline{\chi}\right)}\left(-L'\left(1-s,\overline{\chi}\right)+\frac{\Delta'}{\Delta}\left(1-s,\overline{\chi}\right)L\left(1-s,\overline{\chi}\right)\right).\nonumber
\end{eqnarray}
Moreover, by formula (\ref{eqqL1}), we get $\frac{\xi'}{\xi}(s,\chi)=-\frac{\xi'}{\xi}\left(1-s,\overline{\chi}\right).$ Hence
\begin{eqnarray}I_3&=&-\frac{1}{2i\pi}\int_{1-\delta+iC}^{1-\delta+iT}-\frac{\xi'}{\xi}\left(1-s,\overline{\chi}\right)\left(-L'\left(1-s,\overline{\chi}\right)+\frac{\Delta'}{\Delta}\left(1-s,\overline{\chi}\right)L\left(1-s,\overline{\chi}\right)\right)\frac{X^s}{\Delta\left(1-s,\overline{\chi}\right)}ds\nonumber\\
&=&\frac{1}{2\pi}\int_{C}^{T}\frac{\xi'}{\xi}\left(\delta-it,\overline{\chi}\right)\left(-L'\left(\delta-it,\overline{\chi}\right)+\frac{\Delta'}{\Delta}\left(\delta-it,\overline{\chi}\right)L\left(\delta-it,\overline{\chi}\right)\right)\frac{X^{1-\delta+it}}{\Delta\left(\delta-it,\overline{\chi}\right)}dt.\nonumber
\end{eqnarray}
By complex conjugation, we obtain \begin{eqnarray}\overline{I_3}&=&\frac{X^{1-\delta}}{2\pi}\int_{C}^{T}\frac{\xi'}{\xi}\left(\delta+it,\chi\right)\left(-L'\left(\delta+it,\chi\right)+\frac{\Delta'}{\Delta}\left(\delta+it,\chi\right)L\left(\delta+it,\chi\right)\right)\frac{X^{-it}}{\Delta\left(\delta+it,\chi\right)}dt\nonumber\\
&=&\frac{X^{1-\delta}}{2\pi}\int_{C}^{T}\left\{\frac{1}{2}\log\left(\frac{qt}{2\pi}\right)+\frac{i\pi}{4}
-\sum_{n=1}^{\infty}\frac{\chi(n)\Lambda(n)}{n^{\delta+it}}+O\left(\frac{\log t}{t}\right)\right\}\times\nonumber\\
&&\ \ \times\ \
\left\{\sum_{m=1}^{\infty}\frac{\chi(m)\log{(m)}}{m^{\delta+it}}-\log\left(\frac{qt}{2\pi}\right)\sum_{m=1}^{\infty}\frac{\chi(m)}{m^{\delta+it}}+O\left(\frac{\log t}{t}\right)\right\}\frac{X^{-it}}{\Delta(\delta+it,\chi)}dt.\nonumber
\end{eqnarray}
Let us write $\overline{I_3}$ as follows
\begin{eqnarray}
&&\overline{I_3}=\frac{X^{1-\delta}}{2\pi}\int_{C}^{T}\frac{1}{2}\log{\left(\frac{q
t}{2\pi}\right)}\sum_{m=1}^{\infty}\frac{\chi(m)\log{(m)}}{m^{\delta+it}}\left(\frac{qt}{2\pi}\right)^{\delta-\frac{1}{2}}e^{-\frac{i\pi}{4}}e^{it\log\left(\frac{qt}{2\pi
e X}\right)}dt\nonumber\\
&&-\frac{X^{1-\delta}}{2\pi}\int_{C}^{T}\frac{1}{2}\log^2{\left(\frac{q
t}{2\pi}\right)}\sum_{m=1}^{\infty}\frac{\chi(m)}{m^{\delta+it}}\left(\frac{qt}{2\pi}\right)^{\delta-\frac{1}{2}}e^{-\frac{i\pi}{4}}e^{it\log\left(\frac{qt}{2\pi
e X}\right)}dt\nonumber\\
&&+\frac{X^{1-\delta}}{2\pi}\int_{C}^{T}\frac{i\pi}{4}\sum_{m=1}^{\infty}\frac{\chi(m)\log{(m)}}{m^{\delta+it}}\left(\frac{qt}{2\pi}\right)^{\delta-\frac{1}{2}}e^{-\frac{i\pi}{4}}e^{it\log\left(\frac{qt}{2\pi
e X}\right)}dt\nonumber\\
&&-\frac{X^{1-\delta}}{2\pi}\int_{C}^{T}\frac{i\pi}{4}\log{\left(\frac{q
t}{2\pi}\right)}\sum_{m=1}^{\infty}\frac{\chi(m)}{m^{\delta+it}}\left(\frac{qt}{2\pi}\right)^{\delta-\frac{1}{2}}e^{-\frac{i\pi}{4}}e^{it\log\left(\frac{qt}{2\pi
e X}\right)}dt\nonumber\\
&&-\frac{X^{1-\delta}}{2\pi}\int_{C}^{T}\sum_{n=1}^{\infty}\frac{\chi(n)\Lambda(n)}{n^{\delta+it}}\sum_{m=1}^{\infty}\frac{\chi(m)\log(m)}{m^{\delta+it}}\left(\frac{qt}{2\pi}\right)^{\delta-\frac{1}{2}}e^{-\frac{i\pi}{4}}e^{it\log\left(\frac{qt}{2\pi
e X}\right)}dt\nonumber\\
&&+\frac{X^{1-\delta}}{2\pi}\int_{C}^{T}\sum_{n=1}^{\infty}\frac{\chi(n)\Lambda(n)}{n^{\delta+it}}\sum_{m=1}^{\infty}\frac{\chi(m)}{m^{\delta+it}}\log{\left(\frac{qt}{2\pi}\right)}\left(\frac{qt}{2\pi}\right)^{\delta-\frac{1}{2}}e^{-\frac{i\pi}{4}}e^{it\log\left(\frac{qt}{2\pi
e X}\right)}dt  +O\left(\log^3{T}\right)\nonumber\\
&&=H_1+H_2+H_3+H_4+H_5+H_6+O\left(\log^3{T}\right)\nonumber
\end{eqnarray}
Each of the above integrals will be  evaluated by the  method used by Gonek \cite[Lemma 5, page 131]{9} or Levinson \cite{15}.
First, we have   \footnote{The estimation of $H_{1}$ is based on the calculation of the integral $j_{m}$
$$j_{m}=\int_{C}^{T}g(t)e^{i2\pi f(t)}dt,\ g(t)=  \left(\frac{qt}{2\pi}\right)^{\delta-\frac{1}{2}},\ f(t)=\frac{t}{2\pi}\log\left(\frac{qt}{2\pi e X m}\right).$$
The saddle point $t_{0}=\frac{2\pi Xm}{q}$ belongs to the segment $C\leq t\leq T$ for $m\leq \frac{qt}{2\pi X}$. For such $m$, the main term of the asymptotic formula of  $j_{m}$ has the form
$$e^{i\pi/4}\frac{g(t_{0})e^{2\pi f(t_{0})}}{\sqrt{f''(t_{0})}}=2\pi  e^{i\pi/4} \frac{(Xm)^{\delta}}{\sqrt{q}}\log(mX)e^{-2\pi iXm/q}.$$
}
\begin{eqnarray}
H_1&=&\frac{X^{1-\delta}}{2\pi}\frac{1}{2}\sum_{m=1}^{\infty}\frac{\chi(m)\log{(m)}}{m^{\delta}}\int_{C}^{T}\log{\left(\frac{q
t}{2\pi}\right)}\left(\frac{qt}{2\pi}\right)^{\delta-\frac{1}{2}}e^{-\frac{i\pi}{4}}e^{it\log\left(\frac{qt}{2\pi
e X m}\right)}dt\nonumber\\
&=&\frac{X}{2\sqrt{q}}\sum_{m\leq \frac{q T}{2\pi X}}\chi(m)\log(m)\log(m
X)e^{-2i\pi m X/q}+O\left(\sqrt{T}\log^3{T}\right)\nonumber\\
&=&\frac{X}{2\sqrt{q}}\sum_{m\leq \frac{q T}{2\pi
X}}\chi(m)\log^2(m)e^{-2i\pi m X/q}+\frac{X\log(X)}{2\sqrt{q}}\sum_{m\leq
\frac{q T}{2\pi X}}\chi(m)\log(m)e^{-2i\pi m X/q}\nonumber\\
&&\ \ +\ \ O\left(\sqrt{T}\log^3{T}\right).\nonumber
\end{eqnarray}
To estimate $H_{2}$ we proceed as follows
\begin{eqnarray}
H_2&=&-
\frac{1}{2}\frac{X^{1-\delta}}{2\pi}\sum_{m=1}^{\infty}\frac{\chi(m)}{m^{\delta}}\int_{C}^{T}\log^2{\left(\frac{q
t}{2\pi}\right)}\left(\frac{qt}{2\pi}\right)^{\delta-\frac{1}{2}}e^{-\frac{i\pi}{4}}e^{it\log\left(\frac{qt}{2\pi
e X m}\right)}dt\nonumber\\
&=&-\frac{X}{2\sqrt{q}}\sum_{m\leq \frac{q T}{2\pi X}}\chi(m)\log^2(m
X)e^{-2i\pi m X/q}+O\left(\sqrt{T}\log^3{T}\right)\nonumber\\
&=&-\frac{X}{2\sqrt{q}}\sum_{m\leq \frac{q T}{2\pi
X}}\chi(m)\log^2(m)e^{-2i\pi m X/q}-\frac{X\log^2{(X)}}{2\sqrt{q}}\sum_{m\leq
\frac{q T}{2\pi X}}\chi(m)e^{-2i\pi m X/q}\nonumber\\
&&\ \ -\ \ \frac{X\log{(X)}}{\sqrt{q}}\sum_{m\leq \frac{q T}{2\pi
X}}\chi(m)\log(m)e^{-2i\pi mX/q}+O\left(\sqrt{T}\log^3{T}\right).\nonumber
\end{eqnarray}
For $H_{3}$, we have
\begin{eqnarray}
H_3&=&\frac{X^{1-\delta}}{2\pi}\frac{i\pi}{4}\sum_{m=1}^{\infty}\frac{\chi(m)\log{(m)}}{m^{\delta}}\int_{C}^{T}\left(\frac{qt}{2\pi}\right)^{\delta-\frac{1}{2}}e^{-\frac{i\pi}{4}}e^{it\log\left(\frac{qt}{2\pi
e X m}\right)}dt\nonumber\\
&=&\frac{i\pi}{4}\frac{X}{\sqrt{q}}\sum_{m\leq \frac{q T}{2\pi
X}}\chi(m)\log(m)e^{-2i\pi mX/q}+O\left(\sqrt{T}\log^3{T}\right).\nonumber
\end{eqnarray}
Similarly, we obtain
\begin{eqnarray}
H_4&=&-\frac{X^{1-\delta}}{2\pi}\frac{i\pi}{4}\sum_{m=1}^{\infty}\frac{\chi(m)}{m^{\delta}}\int_{C}^{T}\log{\left(\frac{q
t}{2\pi}\right)}\left(\frac{qt}{2\pi}\right)^{\delta-\frac{1}{2}}e^{-\frac{i\pi}{4}}e^{it\log\left(\frac{qt}{2\pi
e X m}\right)}dt\nonumber\\
&=&-\frac{i\pi}{4}\frac{X\log{X}}{\sqrt{q}}\sum_{m\leq \frac{q T}{2\pi
X}}\chi(m)e^{-2i\pi m X/q}-\frac{i\pi}{4}X\sum_{m\leq \frac{q T}{2\pi
X}}\chi(m)\log(m)e^{-2i\pi m X/q}\nonumber\\
&&\ \ -\ \ \ O\left(\sqrt{T}\log^3{T}\right).\nonumber
\end{eqnarray}
For $H_{5}$, one has
\begin{eqnarray}
H_5&=&-\frac{X^{1-\delta}}{2\pi}\sum_{n=1}^{\infty}\frac{\chi(n)\Lambda(n)}{n^{\delta}}\sum_{m=1}^{\infty}\frac{\chi(m)\log(m)}{m^{\delta}}\int_{C}^{T}\left(\frac{qt}{2\pi}\right)^{\delta-\frac{1}{2}}e^{-\frac{i\pi}{4}}e^{it\log\left(\frac{qt}{2\pi
e X m n}\right)}dt\nonumber\\
&=&-\frac{X}{\sqrt{q}}\sum_{k\leq \frac{q T}{2\pi
X}}\sum_{k=mn}\Lambda(n)\chi(n)\chi(m)\log(m)e^{-2i\pi kX/q}+O\left(\sqrt{T}\log^3{T}\right).\nonumber
\end{eqnarray}
Finally, for $H_{6}$ we get
\begin{eqnarray}
H_6&=&\frac{X^{1-\delta}}{2\pi}\sum_{n=1}^{\infty}\frac{\chi(n)\Lambda(n)}{n^{\delta}}\sum_{m=1}^{\infty}\frac{\chi(m)}{m^{\delta}}\int_{C}^{T}\log{\left(\frac{qt}{2\pi}\right)}\left(\frac{qt}{2\pi}\right)^{\delta-\frac{1}{2}}e^{-\frac{i\pi}{4}}e^{it\log\left(\frac{qt}{2\pi
e X m n}\right)}dt\nonumber\\
&=&\frac{X\log{X}}{\sqrt{q}}\sum_{k\leq \frac{q T}{2\pi
X}}\sum_{k=mn}\Lambda(n)\chi(n)\chi(m)\log(k)e^{-2i\pi k
X/q}+O\left(\sqrt{T}\log^3{T}\right).\nonumber\\
&=&\frac{X}{\sqrt{q}}\sum_{k\leq \frac{q T}{2\pi X}}\chi(k)\log^2(k)e^{-2i\pi k
X/q}+\frac{X\log{X}}{\sqrt{q}}\sum_{k\leq \frac{q T}{2\pi X}}\chi(k)\log(k)e^{-2i\pi k
X/q}\nonumber\\ &&\ \ +\ \ O\left(\sqrt{T}\log^3{T}\right).\nonumber
\end{eqnarray}
Collecting together the above results on $H_{1},.., H_{5}$ and $H_{6}$, we obtain
\begin{eqnarray}  \label{eqq9}
I_3&=&\frac{X}{\sqrt{q}}\sum_{k\leq \frac{q T}{2\pi
X}}\log^2(k)\overline{\chi}(k)e^{2i\pi k X/q}+\frac{X\log(X)}{2\sqrt{q}}\sum_{k\leq \frac{q T}{2\pi
X}}\log(k)\overline{\chi}(k)e^{2i\pi k X/q}\nonumber\\ &&\ \ -\ \
\left(\frac{X\log^2(X)}{2\sqrt{q}}-\frac{i\pi}{4}\frac{X\log{X}}{\sqrt{q}}\right)\sum_{k\leq
\frac{q T}{2\pi X}}\overline{\chi}(k)e^{2i\pi k X/q}\nonumber\\
&&\ \ -\ \ \frac{X}{\sqrt{q}}\sum_{k\leq \frac{q T}{2\pi
X}}\sum_{k=mn}\Lambda(n)\overline{\chi}(n)\overline{\chi}(m)\log(m)e^{2i\pi k X/q}+ O\left(\sqrt{T}\log^3{T}\right).
\end{eqnarray}
Finally, by using equations (\ref{eqq8}) and (\ref{eqq9}),
we  finish the proof of Lemma \ref{lem.2}.
\end{proof}
\section{Proof of Theorem \ref{theorem}}
 Let $X$ be a fixed positive
real number and $a$ be a complex number. We write $s=\sigma+it,\ \
\rho_{a,\chi}=\beta_{a,\chi}+i\gamma_{a,\chi}$ with real
$\sigma,t,\beta_{a,\chi}$ and $\gamma_{a,\chi}.$ By the theorem of residues (or Cauchy's theorem), we get
\begin{equation}\label{eq4}\sum_{\rho_{a,\chi};\ 0<\gamma_{a,\chi}\leq
T}L'\left(\rho_{a,\chi},\chi\right)X^{\rho_{a,\chi}}=\frac{1}{2i\pi}\int_\mathbf{R}\frac{L'^{2}(s,\chi)}{L(s,\chi)-a}X^sds,\end{equation}
where the integration is taken over a rectangular contour in
counterclockwise direction denoted by $\mathbf{R}$ according to the
location of the nontrivial $a$-points of $L(s,\chi)$ which will be
specified below.  In view of  formula (\ref{eq T2 }), the
ordinates of the $a$-points cannot lie too dense. For any large
$T_0\geq 0,$ we can find a real number $T\in \left[T_0,T_0+1\right[$
such that $\min_{\rho_{a,\chi}}\left|T-\gamma_{a,\chi}\right|\gg\frac{1}{\log{T}}.$
We shall distinguish the case $a\neq1$ and $a=1$. Let us suppose that $a\neq1$. We may choose the
counterclockwise oriented rectangular $\mathbf{R}$ with vertices
$1-b+i,\,B+i,\,B+iT,\,1-b+iT$, where $B$ is a large  constant which will be chosen below and
$b=1+\frac{1}{\log{T}},$ at the expense of a small error for
disregarding  at most finitely many nontrivial $a$-points below
$Im(s)=1$ and for counting finitely many trivial $a$-points to the
right of $Re(s)=1-b.$ Then, we have $$\sum_{\rho_{a,\chi};\ 0<\gamma_{a,\chi}\leq
T}L'\left(\rho_{a,\chi},\chi\right)X^{\rho_{a,\chi}}=\frac{1}{2i\pi}\int_{\mathbf{R}}\frac{L'^{2}(s,\chi)}{L(s,\chi)-a}X^sds+O(1).$$ Hence
\begin{eqnarray*}
\sum_{\rho_{a,\chi};\ 0<\gamma_{a,\chi}\leq
T}&L'\left(\rho_{a,\chi},\chi\right)&X^{\rho_{a,\chi}}=\frac{1}{2i\pi}\left\{\int_{B+i}^{B+iT}+\int_{B+iT}^{1-b+iT}+\int_{1-b+iT}^{1-b+i}\right\}\frac{L'^{2}(s,\chi)}{L(s,\chi)-a}X^{s}ds\\
&&+\frac{1}{2i\pi}\int_{1-b+i}^{B+i}\frac{L'^{2}(s,\chi)}{L(s,\chi)-a}X^{s}ds
+O(1)=I_1+I_2+I_3+I_4+O(1).
\end{eqnarray*}
It is easy to see from equation \eqref{eq.l'/l} that $I_2,I_4\ll T^{\frac{1}{2}+\epsilon}.$  Now, let us estimate the two integrals $I_1$ and $I_3$. Recall that,  for $\sigma\longrightarrow \infty$, we have $L(s,\chi)=1+o(1)$ and $ L'(s,\chi)\ll 2^{-\sigma} $ uniformly in $t.$ Hence, there are no $a$-points for sufficiently large $\sigma$ provided that $a\neq1$. For the case $a=1$, we define $m=\min\{n\geq2;\chi(n)\neq0\}$. We observe, for $\sigma\longrightarrow +\infty$, $L(s,\chi)-1=\frac{\chi(m)}{m^{\sigma+it}}(1+o(1))$. Hence, in both cases, $a\neq0$ or $a=1$, we choose $B$ a fixed constant sufficiently large such that there are no $a$-points of $L(s,\chi)$ in the half-plane $\sigma> B-1$ (see \cite[Equations (20) and (21)]{8}). Therefore,   we deduce that
$$\int_{B}^{B+iT}\frac{L'^{2}(s,\chi)}{L(s,\chi)-a}X^{s}ds\ll X^B T^{-2\log{2}}\log{T}\ll T^{-2\log{2}}\log{T}.$$
Then $$\sum_{\rho_{a,\chi};\ 0<\gamma_{a,\chi}\leq
T}L'\left(\rho_{a,\chi},\chi\right)X^{\rho_{a,\chi}}=-\frac{1}{2i\pi}\int_{1-b}^{1-b+iT}\frac{L'^{2}(s,\chi)}{L(s,\chi)-a}X^{s}ds+O\left(T^{\frac{1}{2}+\epsilon}\right).$$   It remains to evaluate the integral over the left vertical line segment
$[1-b,1-b+iT]$ of the rectangular $\mathbf{R}.$ Using the same argument as above, for $s\in{[1-b+iA,1-b+iT]}$ where $A=A(a)>0,$ we get the geometric series
expansion
$$\frac{L'(s,\chi)}{L(s,\chi)-a}=\frac{L'}{L}\left(s,\chi\right)\left\{1+\frac{a}{L\left(s,\chi\right)}+\sum_{k=2}^{\infty}\left(\frac{a}{L(s,\chi)}\right)^k\right\}.$$
 Since integration over $[1-b+i,1-b+iA]$ yields to a bounded error term,
the integral $I_3$ becomes
\begin{eqnarray*}&&I_3\nonumber\\
&=&\frac{1}{2i\pi}\int_{1-b+iT}^{1-b+iA}\left\{\frac{L'^2}{L}(s,\chi)X^s+a\left(\frac{L'}{L}(s,\chi)\right)^2X^s+\frac{L'^2}{L}(s,\chi)\sum_{k=2}^{+\infty}\left(\frac{a}{L(s,\chi)}\right)^{k}X^s\right\}ds+O(1)\\
&=&J_1+J_2+J_3+O(1).
\end{eqnarray*}
In order to estimate the third integral $J_3,$ we use equations \eqref{eq.moment1} and \eqref{eq.daven} (similar computation was done in \cite[integral ${\mathcal J}_{3}$ page 30]{8}) to obtain
\begin{eqnarray*}J_3&=&-\frac{aX^{1-b}}{2\pi}\int_{A}^{T}\left(\frac{L'}{L}(1-b+it,\chi)\right)^2\sum_{l=1}^{+\infty}\left(\frac{a}{L(1-b+it,\chi)}\right)^lX^{it}dt\\
&\ll& T^{\frac{1}{2}+\epsilon},\ \ \ \epsilon> 0.
\end{eqnarray*}
Next, let us consider the integral $J_2.$ By the  functional equation satisfied by $L(s,\chi)$, we write the integral
$J_2$ as  \begin{eqnarray*}
J_2&=&-\frac{a}{2i\pi}\int_{1-b+iA}^{1-b+iT}\left(\frac{\Delta'}{\Delta}(s,\chi)-\frac{L'}{L}\left(1-s,\overline{\chi}\right)\right)^2X^sds\\
&=&-\frac{a}{2i\pi}\int_{1-b+iA}^{1-b+iT}\left\{\left(\frac{\Delta'}{\Delta}(s,\chi)\right)^2-2\frac{\Delta'}{\Delta}(s,\chi)\frac{L'}{L}\left(1-s,\overline{\chi}\right)+\left(\frac{L'}{L}\left(1-s,\overline{\chi}\right)\right)^2\right\}X^sds\\
&=&N_1+N_2+N_3.
\end{eqnarray*}
We have
\begin{eqnarray*}
N_1&=&-\frac{a}{2\pi}X^{1-b}\int_A^T\left(-\log\left(\frac{qt}{2\pi}\right)+O\left(\frac{1}{t}\right)\right)^2X^{it}dt\\
&=&-\frac{a}{2\pi}X^{1-b}\int_A^T\left(\log^2\left(\frac{qt}{2\pi}\right)+O\left(\frac{\log\left(\frac{qt}{2\pi}\right)}{t}\right)+O\left(\frac{1}{t}\right)\right)X^{it}dt\\
&=&-\frac{aT}{2\pi}\log^{2}\frac{q
T}{2\pi}+\frac{aT}{\pi}\log\frac{q
T}{2\pi}-\frac{aT}{\pi}+O\left(\log^{2}T\right).
\end{eqnarray*}
Furthermore
\begin{eqnarray*}
N_2&=&\frac{a}{i\pi}\int_{1-b+iA}^{1-b+iT}\frac{\Delta'}{\Delta}(s,\chi)\frac{L'}{L}\left(1-s,\overline{\chi}\right)X^sds\\
&=&\frac{a}{\pi}\int_{A}^{T}\frac{\Delta'}{\Delta}(1-b+it,\chi)\frac{L'}{L}\left(b-it,\overline{\chi}\right)X^{1-b+it}dt\\
&=&\frac{a}{\pi}\int_{A}^{T}\left(\log\left(\frac{qt}{2\pi}\right)+O\left(\frac{1}{t}\right)\right)\sum_{n=2}^{\infty}\frac{\overline{\Lambda}(n)\overline{\chi}(n)}{n^{b-it}}X^{1-b+it}dt\\
&=&\frac{a}{\pi}X^{1-b}\sum_{n=2}^{\infty}\frac{\overline{\Lambda}(n)\overline{\chi}(n)}{n^{b}}\int_{A}^T\left(\log\left(\frac{qt}{2\pi}\right)+O\left(\frac{1}{t}\right)\right)e^{it\log\left({Xn}\right)}dt.
\end{eqnarray*}
An integration by parts yields to
$$\int_{A}^T\left(\log\left(\frac{qt}{2\pi}\right)+O\left(\frac{1}{t}\right)\right)e^{it\log\left({Xn}\right)}dt=O\left(\log T\right).$$
So  $$N_2\ll
\log{T}\sum_{n=2}^{\infty}\frac{\overline{\Lambda(n)\chi(n)}}{n^{b}}X^{1-b}\ll\log{T}\left|\frac{L'}{L}\left(b,\overline{\chi}\right)\right|X^{1-b}\ll
\log^3 T.$$
The same argument gives
\begin{eqnarray*}
N_3&=&-\frac{a}{2i\pi}\int_{1-b+iA}^{1-b+iT}\left(\frac{L'}{L}\left(1-s,\overline{\chi}\right)\right)^2X^sds   \nonumber\\
&=&-\frac{a}{2\pi}\int_{A}^{T}\left(\frac{L'}{L}\left(b-it,\overline{\chi}\right)\right)^2X^{1-b+it}dt\\
&=&-\frac{a}{2\pi}X^{1-b}\sum_{m,n=2}^{+\infty}\frac{\overline{\Lambda}(m)\overline{\Lambda}(n)}{(m
n)^b}\int_A^Te^{it\log(Xmn)}dt \\
&\ll&    \sum_{m,n=2}^{+\infty}\frac{\overline{\Lambda}(m)\overline{\Lambda}(n)}{(m
n)^b}\left|\int_A^Te^{it\log(Xmn)}\right|dt\ll \log^2{T}.
\end{eqnarray*}
From the above estimations of $N_1,\ N_2$ and $N_3$, we obtain $$J_2=-\frac{aT}{2\pi}\log^{2}\frac{q
T}{2\pi}+\frac{aT}{\pi}\log\frac{q
T}{2\pi}-\frac{aT}{\pi}+O\left(\log^{3}T\right).$$
Finally,  to end the proof of Theorem \ref{theorem}, it remains to evaluate $J_{1}$ given by
$$J_{1}=\frac{1}{2i\pi}\int_{1-b+iT}^{1-b+iA}\frac{L'^2}{L}(s,\chi)X^sds$$
which can be evaluated, firstly, up to an error term as a contour integral
$$J_{1}=\frac{1}{2i\pi}\int_{R}\frac{L'^2}{L}(s,\chi)X^sds+O\left(T^{1/2+\epsilon}\right)$$
and, secondly, as a sum of residues
$$J_{1}=\sum_{0<\gamma_{a,\chi}<T}L'(\rho_{\chi},\chi)X^{\rho_{\chi}}+O\left(T^{1/2+\epsilon}\right),$$
where $\rho_{\chi}=\beta_{\chi}+i\gamma_{\chi}$ stands for nontrivial zeros of $L(s,\chi)$. To finish the proof of Theorem \ref{theorem} for $a\neq1$, we note that the last sum of residues was evaluated in Lemma \ref{lem.2}.  \\

For $a=1$, we consider the function $l(s)=q^{s}(L(s,\chi)-1)$ in place of $L(s,\chi)-a$. Furthermore, we have $\frac{l'}{l}(s)=\log q+\frac{L'(s,\chi)}{L(s,\chi)-1}.$ This implies that the constant term does not contribute by integration over a closed contour and we use the same argument as in the case $a\neq1$.

\section{Concluding remarks}
We believe that we can extend our result to higher derivatives of $L(s,\chi)$. The $a$-points of an $L$-function $L(s)$   are the roots of the equation $L(s) = a$. We refer to Steuding book \cite[chapter 7]{21} for some results about $a$-points of  $L$-functions from the Selberg class. Therefore, it is an interesting question to extend Theorem \ref{theorem} to other classes of Dirichlet $L$-functions (the Selberg class with some further conditions) and its higher derivatives and to refine the error term in Theorem \ref{theorem} under the Riemann hypothesis. These problems will be considered in a sequel to this article.\\

\noindent  {\bf Acknowledgement.}  We are thankful to the anonymous referee for his/her valuable suggestions which helped us to improve  the paper.


\begin{thebibliography}{99}
\bibitem[1]{1} H. Davenport, {\it Multiplicative number theory,} Springer 1980, 2nd ed. revised by H. L. Montgomery.
\bibitem[2]{2} A. Fujii, {\it On a conjecture of Shanks,} Proc. Japan Acad. Ser. A Math. Sci, 70(4): 109-114, 1994.
\bibitem[3]{3} A. Fujii, {\it On the distribution of values of derivative of the Riemann zeta function at its zeros.I,} Proceedings of the Steklov Institute of Mathematics, Vol. 276, pp. 51-76. 2012.
\bibitem[4]{4} A. Fujii, {\it Some observations concerning the distribution of the zeros of the Zeta functions. II,} Comment. Math. Univ. St. Pauli {\bf40} (2) (1991), 125–-231.
\bibitem[5]{5} A. Fujii, {\it  Uniform distribution of the zeros of the Riemann zeta function and the mean value theorems of Dirichlet $L$-functions}, Proc. Japan Acad. {\bf63A} (1987), 370--373.
\bibitem[6]{6}  A. Fujii, {\it Zeta zeros and Dirichlet $L$-functions,} Proc. Japan Acad. Ser. A Math. Sci. {\bf 64} ( 6) (1988), 215--218.

\bibitem[7]{7} R. Garunk$\check{s}$tis and J. Steuding, {\it  On the roots of the equation $\zeta(s) = a,$} Abh. Math. Semin. Univ. Hambg. {\bf 84} (2014), 1--15.
\bibitem[8]{8} R. Garunk$\check{s}$tis, J. Grahl and J. Steuding, {\it Uniqueness Theorem for
$L$-functions}, Commentarii Mathematici Universitatis Sancti Pauli {\bf60} (1-2) (2011), 15--35.
\bibitem[9]{9} S. M. Gonek, {\it Mean values of the Riemann zeta-function and its derivatives,} Invent. Math.  {\bf75}  (1984), 123--141.
\bibitem[10]{10} S. M. Gonek, S. J. Lester and M. B. Milinovich, {\it A note on simple $ a$-points of $ L$-functions,} Proceeding of the AMS {\bf140} (12) (2012), 4097-–4103.
\bibitem[11]{11} G. A. Hiary and A. M. Odlyzko, {\it Numerical study of the derivative of the Riemann zeta function at zeros,}  ArXiv:1105.4312.
\bibitem[12]{12} M-T. Jakhlouti and K. Mazhouda, {\it Distribution of the values of the derivative of the Riemann zeta function at its $a$-points}, Uniform Distribution Theory Journal {\bf9} (1) (2014), 115--125.
 \bibitem[13]{13}   A.F. Lavrik, {\it The  approximate functional equation for Dirichlet $L$-functions,} Tr. Mosk. Mat. Obs. {\bf18} (1968), 91--104.
\bibitem[14]{14} N. Levinson, {\it Almost all roots of $\zeta(s) = a$ are arbitrarily close to $\sigma =1/2$,} Proc. Nat. Acad. Sci. U.S.A. {\bf72} (1975), 1322--1324.
\bibitem[15]{15} N. Levinson, {\it More than one third of zeros of Rieman's zeta function are one $\sigma=1/2$,} Adv. Math. {\bf 13} (1974), 383--436.
\bibitem[16]{16} K. Mazhouda and S. Omar, {\it Mean-square of $L$-functions in the Selberg class,} Steuding, Rasa  et al. (ed.), New directions in value-distribution theory of zeta and $L$-functions. Proceedings of the conference, W$\ddot{u}$rzburg, Germany, October 6-10, 2008. Aachen: Shaker Verlag. (2009),  249--263.
    \bibitem[17]{17} T. Onozuka {\it On the $a$-points of the derivatives of the Riemann zeta,} European Journal of Mathematics  {\bf 3} (1) (2017),  53-–76.
\bibitem[18]{18} K. Prachar, {\it Primzahlverteilung,} Springer-Verlag (Wien, 1957).
\bibitem[19]{19} V. Rane, {\it On an approximate functional equation for Dirichlet $L$-series}, Math. Ann. {\bf264}
137-145 (1983), 137--145.
\bibitem[20]{20} A. Selberg, {\it Old and new conjectures and results about a class of Dirichlet series,} in: Pro-
ceedings of the Amalfi Conference on Analytic Number Theory, Maiori 1989, E. Bombieri et
al. (eds.), Universit\`a di Salerno (1992), 367--385.
\bibitem[21]{21} J. Steuding, {\it Value-distribution of $L$-functions,} Lecture Notes in Mathematics 1877, Springer
2007, 126-127.

\end{thebibliography}
\end{document}